\documentclass[3p]{elsarticle}

\usepackage{amsmath,amssymb,amsthm}
\usepackage{color}
\usepackage{graphicx}
\usepackage{geometry}
\usepackage[normalem]{ulem}

\usepackage{framed}

\journal{Not yet submitted}

\definecolor{darkgreen}{rgb}{0.0,0.7,0.0}

\definecolor{darkred}{rgb}{0.75,0.0,0.0}

\newcommand{\R}{\mathbb{R}}

\newcommand{\DD}{\mathcal{D}}

\DeclareMathOperator{\tr}{tr\,}
\DeclareMathOperator*{\argmax}{arg\,max}

\theoremstyle{plain}
\newtheorem{thm}{Theorem}[section]
\newtheorem{lem}[thm]{Lemma}

\theoremstyle{definition}
\newtheorem{defn}{Definition}[section]

\newtheorem{rem}{Remark}[section]

\begin{document}

\begin{frontmatter}

\title{Self-oscillations in an Alpha Stirling Engine: \\ a bifurcation analysis}

\author[auth1]{Dmitry Gromov\corref{corrauth}}
\cortext[corrauth]{Corresponding author}
\ead{dv.gromov@gmail.com}

\author[auth2]{Fernando Casta\~nos}
\ead{castanos@ieee.org}

\address[auth1]{Faculty of Applied Mathematics and Control Processes, St. Petersburg State University, Russia}
\address[auth2]{Departamento de Control Autom\'atico, Cinvestav-IPN, 2508 Av. Instituto 
Polit\'ecnico Nacional, 07360, Mexico City, Mexico}

\begin{abstract}
 We study a thermo-mechanical system comprised by an alpha Stirling engine and a flywheel from the perspective of dynamical systems theory. Thermodynamics establish a static relation between the flywheel's angle and the forces exerted by the two power pistons that
 constitute the engine. Mechanics, in turn, provide a dynamic relation between the forces and the
 angle, ultimately leading to a closed dynamical model. We are interested in the different behaviors
 that the engine displays as parameters are varied. The temperature of the hot piston and 
 the mechanical phase between both pistons constitute our bifurcation parameters. Considering that 
 energy conversion in the engine can only take place through cyclic motions, we are particularly
 interested in the appearance of limit cycles.
\end{abstract}


\end{frontmatter}
\section{Introduction}

The Stirling engine is an external combustion thermodynamic engine that operates by cyclic expansion and contraction of a working fluid (typically air). Due to their high efficiency and capability of operating at low temperatures and with any heat sources, Stirling engines have found many applications, ranging from electricity generators~\cite{Schreiber:07} and cryocoolers~\cite{Filis:09} to solar power generators~\cite{Kong:03,Tlili:08,Abbas:08} and combined heat and power systems~\cite{Onov:06,Alanne:10,Ferreira:16}. 

A Stirling engine comprises two chambers. The chambers, not necessarily separated physically, are connected through a regenerator and, at the same time, mechanically through a load. The operation of a Stirling engine consists of cyclic heat absorption and discharge, accompanied by mass transfer between the chambers and, consequently, oscillations in the internal energy.  Through a mechanical coupling, there is a continuous exchange of energy between the flywheel and the heat engine. The net effect is producing useful work that can be either stored or transformed into electrical energy.

In this work, we consider the {\em alpha} configuration of the Stirling engine. It consists of two separate cylinders and two power pistons, one in a hot cylinder, one in a cold cylinder. Both pistons are connected to the flywheel in such a way that there is a phase shift in the movement of the pistons. Such phase shift is denoted by $\alpha$. When describing the thermodynamic state of the Stirling engine, we use the classical isothermal Schmidt model~\cite{Schmidt:1871}, which provides reasonable accuracy for a wide range of operating conditions~\cite{Berchowitz:84,Walker:80}. 
The isothermal assumption allows establishing a simple expression between the gas pressure and the cylinders' volumes which, in turn, depend on the positions of the pistons. The latter are determined by the mechanics of the crank and the angle of rotation of the flywheel. The interconnection of the mechanical and thermodynamic components results in a highly nonlinear system whose state is given by the angular position and the velocity of the flywheel.

There are several papers on dynamic modeling of Stirling engines, see, e.g.,~\citep{Schulz:96,Cheng:11,Hooshang:16}, as well as various findings that result from analyzing the engines from a control-theoretic viewpoint: dynamic analysis of a periodically controlled Stirling engine~\cite{CraunBamieh:15,HauserBamieh:15}, local analysis of a controlled free-piston Stirling engine and the identification of the conditions under which oscillations may occur~\cite{Riofrio:08}, linear analysis of a wobble-yoke Stirling engine~\cite{Alvarez:10}, and a control-geometric approach to the description of a Stirling engine~\cite{MuellerCaines:15}. However, to the best of authors' knowledge, no systematic parametric analysis of the Stirling engine dynamic has been undertaken so far.

The continued periodic operation of a Stirling engine under a wide range of varying conditions relies upon the existence of a stable limit cycle, which is typically visualized either in pressure--volume ($p-V$) or in temperature--entropy ($T-S$) variables. The standard thermodynamic analysis is typically carried out under the assumption that such a limit cycle exists. Whilst this assumption is empirically reasonable when the engine operates at high temperatures, it is not clear up to which point is this assumption rational at low temperatures. Motivated by the need for a more formal understanding of thermodynamic cycles,  we take a dynamical systems perspective and investigate the mechanism under which the system transitions from the non-existence to the existence of limit cycles. One of our findings is that the transition takes place through a global bifurcation that, similar to the Andronov-Leontovich bifurcation~\cite{Andronov:59}, involves the brief existence of a homoclinic orbit.

A particular feature of the system model is that it has a cylindrical phase space. The topological difference between planar and cylindrical phase spaces has more 
    implications than it may seem. On the cylinder, for example, not every closed curve can be continuously shrunk to a point. As a consequence, Bendixson's criterion becomes false in general. Also, when written in local coordinates, continuity of the vector field imposes the periodicity of the equations describing it. Periodicity is a form of symmetry that lowers the codimension of some bifurcations. In particular, our model exhibits pitchfork bifurcations which are only of codimension one. Finally, the polynomial normal forms that are commonly used to identify bifurcations are no longer useful, since they are not periodic. The absence of essential pieces of analytic machinery that exists for planar systems  forces us to undertake the first approach to our problem from a numerically oriented perspective.

When carrying out the analysis, we are particularly interested in the qualitative behavior of the system as the temperature of the heat source, $T_h$, changes. We are also interested in the effects of changing $\alpha$. The role of $\alpha$ has been the subject of long-standing debates and, so far, its best value is usually determined experimentally. We hope that the analysis carried out will contribute to the establishment of more rational guidelines for choosing its value in applications.

Besides the bifurcations already mentioned, there are bifurcations of codimension two in which non-hyperbolic equilibria coexist with homoclinic and heteroclinic orbits, or where homoclinic and heteroclinic orbits appear simultaneously. The bifurcations of codimension two organize the $(\alpha,T_h)$-plane into eight distinct regions such that, within each region, every pair of parameters yields a topologically equivalent phase plane. Only four of these regions correspond to phase planes containing a stable limit cycle, and only these regions are suitable for energy conversion.
    
Also on the $(\alpha,T_h)$-plane, we compute level curves for the output power. Computing the power for a given pair of parameters requires the detection of \emph{the} limit cycle (uniqueness is established below), including its period. From the curves, it is possible to determine, for a fixed $T_h$, the value of $\alpha$ that yields the maximal output power. It is also possible, e.g., to estimate the minimal working temperature of the hot piston.

The paper is organized as follows. The dynamic model is derived in Section~\ref{sec:model}. Existence, uniqueness and possible types of limit cycles are discussed in Section~\ref{sec:qualitative}. Section~\ref{sec:numerical} contains the results of the numerical analysis. The occurrence of local and global bifurcations is shown in parameter space and the output power is computed for several parameter pairs. Conclusions and future work are stated at the end of the paper.
\section{System description and dynamical model} \label{sec:model}

\subsection{General principles of operation} 

In the following, we consider the {\em alpha} configuration of a Stirling engine. This particular configuration has two communicating cylinders and two power pistons, both connected to a flywheel as shown in Fig.\ \ref{fig:mechanical}. The first cylinder is in thermal contact with a hot bath (an infinite source) kept at temperature $T_h$, and the second one is attached to a cold bath (an infinite sink) at temperature $T_c$. The gas is moved between the cylinders by the respective pistons. 

\begin{figure}
\centering
\includegraphics[width=0.5\columnwidth]{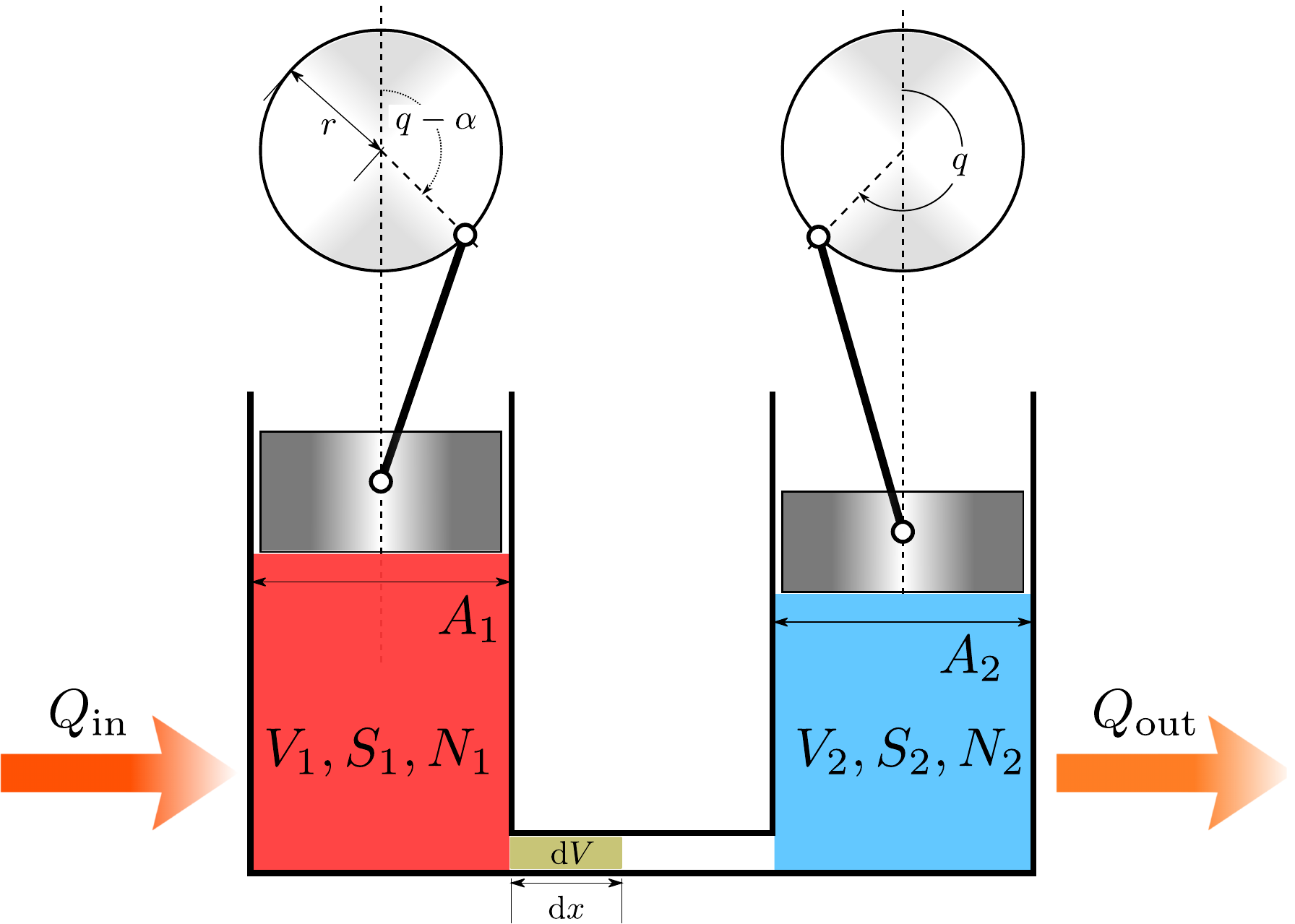}
\caption{Alpha Stirling engine.}
\label{fig:mechanical}
\end{figure}

Since the pistons are connected to the flywheel with a phase shift $\alpha$, their movements produce different effects on the gas contained in the respective cylinders. The whole operation cycle can be separated into four phases as shown in Fig.~\ref{fig:pistons_move}: during the first phase (I) the gas in the second cylinder undergoes compression while the gas in the first cylinder expands; this is followed by compression in both cylinders in phase II; in phase III the gas in the first cylinder continues compressing, while the gas in the second cylinder expands; finally, in the fourth phase (IV) the gas expands in both cylinders. The cyclic operation is accompanied by heat absorption and discharge that occur at different rates, depending on the state of the working fluid within the respective cylinder. The movement of the pistons is accompanied by the mass transfer from the cylinder with higher pressure to the cylinder with lower pressure. The mass transfer equilibrates the pressures, a process that occurs on a fast time scale and can, therefore, be taken to be instantaneous. We can thus assume that the pressure in both cylinders is equal.

\begin{figure}[tbh]
\centering
\includegraphics[width=0.75\columnwidth]{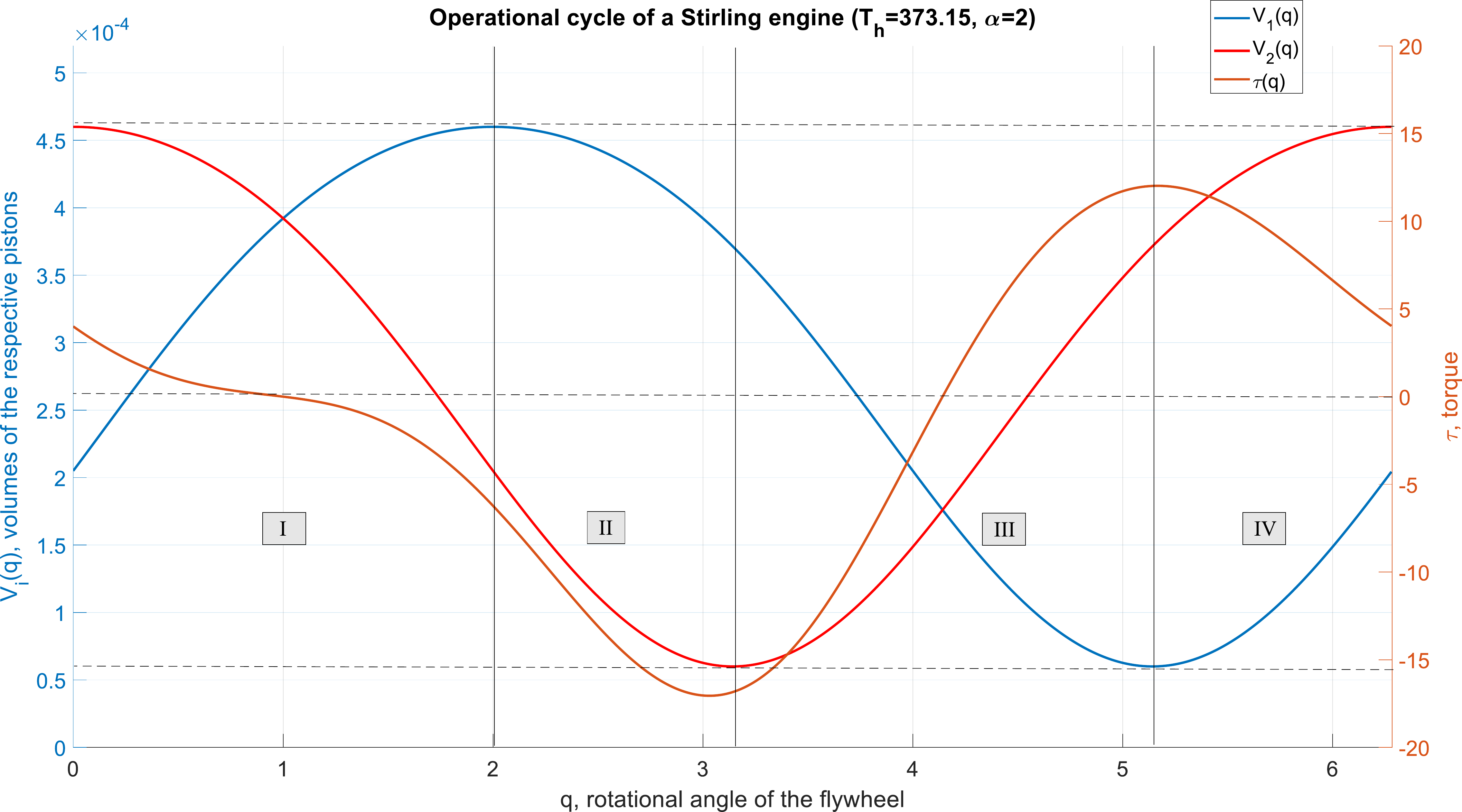}
\caption{Operation cycle of a Stirling engine: four phases.}
\label{fig:pistons_move}
\end{figure}


In the following, we separately discuss the thermodynamic and the mechanical components of a Stirling engine and finally present the coupled model. 

\subsection{Thermodynamic subsystem} 

The thermodynamic component consists of two cylinders filled with an ideal gas and connected by a regenerator.  
We consider the amount of gas within the $i$th cylinder, $i=1,2$, as a homogeneous, single-phase, and single-component thermodynamic system characterized by extensive parameters: volume $V_i$ and molar number $N_i$ as well as intensive parameters: temperature $T_i$ and pressure $p_i$. Similarly, we use $N_r$ and $V_r$ to denote the molar number and the volume of the gas within the regenerator.

There are two main assumptions that we make about the thermodynamic part:
\begin{enumerate}
\item[\bf A1.] The total amount of the substance within the engine is constant, i.e., $N_1+N_2+N_r=N$.
\item[\bf A2.] The pressure is both cylinders is equilibrated, i.e., $p_1=p_2=p$.
\end{enumerate}
While the former assumption is intuitively clear, the latter presents a certain idealization which, however, agrees well with the practice for Stirling engines working at relatively low frequencies. 

%

\paragraph{Isothermal Schmidt's model}

We will make use of the ideal isothermal model~\cite{Berchowitz:84}. Within this framework, each cylinder is assumed to be divided in two sections. The first section of the hot cylinder, where the piston moves, is called the ``expansion'' space. The second section is where the heat transfer occurs and is called the ``heater'' (see Fig.~\ref{fig:iso}). Similarly, the cold cylinder is divided in the ``contraction'' and the ``cooler'' sections. This partition ensures that the heat transfer in either of the pistons occurs through the same area and does not depend on the pistons positions.

The critical assumption is that the gases in the hot and the cold cylinders are kept at constant temperatures, equal to those of the heat source and the sink, respectively. The gas temperature within the regenerator is assumed to change linearly as shown in Fig.~\ref{fig:iso}. Thus, we have $T_1=T_h$ and $T_2=T_c$, and the linearly changing temperature of the regenerator is substituted with its {\em mean effective} temperature $T_r=(T_h-T_c)\left(\ln(T_h)-\ln(T_c)\right)$, \cite{Berchowitz:84}. The isothermal assumption allows establishing a simple relation between the gas pressure inside of the cylinders and the variation of the cylinders' volumes which, in turn, depend on the position of the pistons. 

\begin{figure}[tbh]
\centering
\includegraphics[width=0.5\columnwidth]{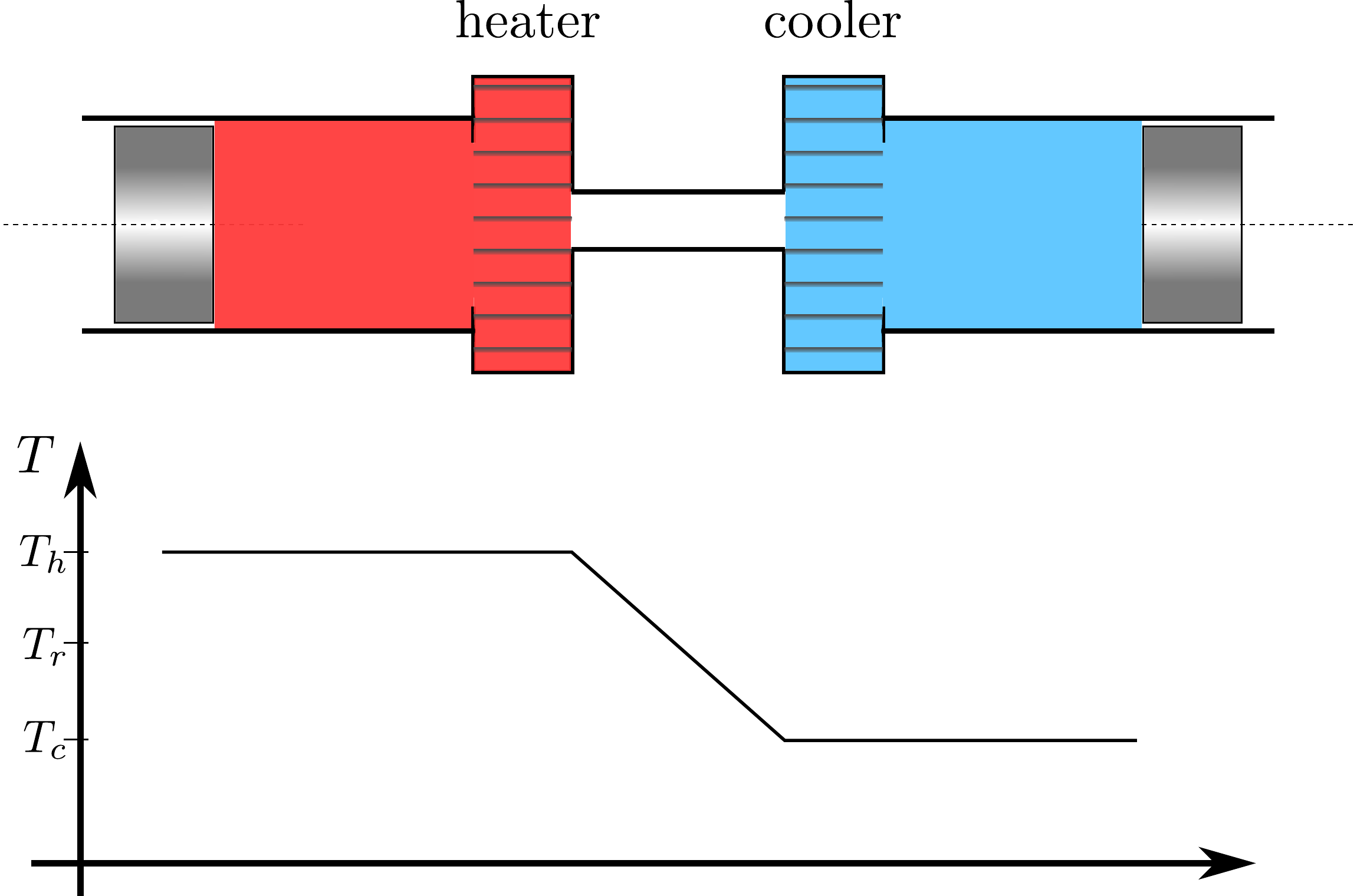}
\caption{Isothermal model.}
\label{fig:iso}
\end{figure}

Recall that, according to \textbf{A1}, the total molar number of gas within the engine is constant, thus:
\begin{equation}\label{eq:N}
 N = N_1 + N_2 + N_r \;.
\end{equation}
Each term on the right-hand side of \eqref{eq:N} can be expressed using the ideal gas law as $N_i = p V_i / R T_i$ with $R$ the universal gas constant. Substituting the respective expressions in~\eqref{eq:N} and using Assumption \textbf{A2} we solve this equation for $p$:
\begin{equation}\label{eq:iso-p0}
 p = \frac{NR}{\left(\frac{V_1}{T_h} + \frac{V_2}{T_c} + \frac{V_r}{T_r}\right)} \;,
\end{equation}
where $V_1$ and $V_2$ correspond to the total volumes of the cylinders. 

Note that both $V_r$ and $T_r$ are constant, hence one can eliminate the term $V_r/T_r$ by increasing the volumes of both 
cylinders by $\delta V = \frac{T_h T_c V_r}{(T_h+T_c)T_r}$. 
We will reduce~\eqref{eq:iso-p0} as
\begin{equation}\label{eq:iso-p}
 p = \frac{NR}{\left(\frac{V_1}{T_h} + \frac{V_2}{T_c}\right)} \;,
 \end{equation}
where, to simplify the notation, $V_1$ and $V_2$ have been renamed to the augmented volumes of the cylinders.

\subsection{Mechanical subsystem}

The mechanical part consists of a flywheel and two power pistons attached to it (see Fig.~\ref{fig:mechanical}). The angular position of the flywheel uniquely determines the linear positions of the pistons, and hence the volumes $V_i$. The volumes are thus functions of the flywheel's angle and the geometry of the cylinders.  

\begin{figure}[tbh]
\centering
\includegraphics[width=0.2\columnwidth]{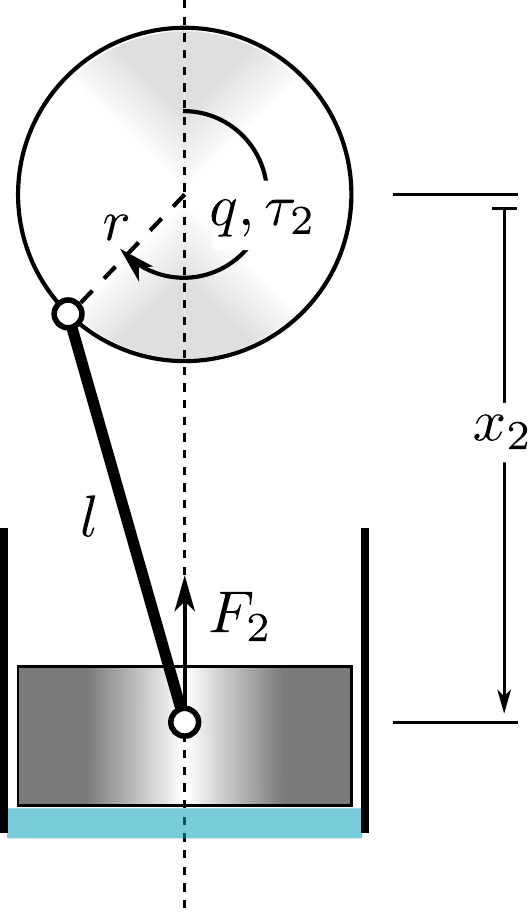}
\caption{Geometry of the flywheel and one of the cranks.}
\label{fig:torque}
\end{figure}

To compute the force developed by the pistons and the torque applied to the flywheel, we consider the second piston as shown
in Fig.~\ref{fig:torque}. Energy conservation imposes the relation 
\begin{equation} \label{eq:Fdx}
 -F_2\delta x_2 = \tau_2 \delta q \;, 
\end{equation}
where $\delta x_2$ and $\delta q$ are infinitesimal linear and angular displacements, respectively. Simple trigonometry establishes that the position of the second piston measured from the center of the flywheel is given by
\begin{displaymath}
 x_2(q) = -r \cos(q) + \sqrt{l^2-r^2\sin^2(q)} \;, 
\end{displaymath}
where $l$ is the length of the rod and $r$ the distance between the center of the flywheel and the mounting point of the shaft. 
%
Writing the infinitesimal displacement $\delta x_2$ in terms of $\delta q$ and substituting it into \eqref{eq:Fdx} gives
\begin{equation}\label{eq:tau2}
 \tau_2 = -F_2 \phi(q) \;, \quad \phi(q) = r\sin(q) - r^2 \frac{\sin(q)\cos(q)}{\sqrt{l^2-r^2\sin^2(q)}} \;.
\end{equation}

Note that the force $F_2$ can also be expressed as $F_2(q) = A_2 (p(q)-p_\mathrm{a})$, where $p(q)$ is given
by~\eqref{eq:iso-p} and~\eqref{eq:Vi}, $p_\mathrm{a}$ is the ambient pressure, and $A_2$ is the cross-section area of the second piston. Since the internal pressure in both cylinders is equal, the torques developed by both pistons depend only on the value of the rotation angle of the flywheel. We can thus write the total torque as
\begin{equation} \label{eq:tau}
 \tau(q) = -\left( A_1 \phi(q-\alpha) + A_2 \phi(q) \right)(p(q)-p_\mathrm{a}) \;.
\end{equation}
The minus sign is consistent with the ``right-hand screw rule'' and the particular choice of the directions of rotation and application of the forces $F_i$.



The volumes of the pistons are
\begin{equation}\label{eq:Vi}
 V_i(q) = V_{i,\mathrm{max}} - A_i \left(x_i(q)- (l-r)\right),\enskip i=1,2,
\end{equation}
where $V_{i,\mathrm{max}}$ is the maximal volume of the $i$th cylinder and $q$ is the angular position of the flywheel. 
Note that, for $q=0$, the position of the respective piston is $l-r$.

Finally, the dynamics of the engine are given by
\begin{equation} \label{eq:model}
 I \ddot{q} = -k_\mathrm{f} \dot{q} + \tau(q) \;,
\end{equation}
where $I$ is the moment of inertia of the flywheel and $k_\mathrm{f}$ is the friction coefficient of its bearings. In the model \eqref{eq:tau}-\eqref{eq:model}, the pressure $p$ serves as the input for the flywheel.

As a quick validation of the model, note that  
\begin{displaymath}
 -(p-p_\mathrm{a})(\dot{V}_1 + \dot{V}_2) = -\left( A_1\phi(q-\alpha) + A_2\phi(q) \right)(p-p_\mathrm{a})\dot{q} = \tau(q)\dot{q} \;.
\end{displaymath}
That is, the thermodynamic work equals the mechanical work performed on the flywheel. 

%


\section{Qualitative analysis} \label{sec:qualitative}

In this section, we prove some properties about the qualitative behavior of system~\eqref{eq:model}. The analysis follows the ideas described in the book~\cite{Barbashin:69}, which is, unfortunately, unavailable in English. We will thus translate the required material when developing the proofs below. 

First, we note that the right-hand side of~\eqref{eq:model} is bounded and jointly smooth in $q$ and $\dot{q}$. This implies that, for any $\bar{t} > 0$, the solutions of~\eqref{eq:model} exist, are unique, and defined on the interval $[0,\bar{t}]$. 
Intuitively, the equations can be identified with those of a rotational mass-spring-damper system with a nonlinear, non-monotone spring. Written in angular coordinates, the total energy of the system is 
\begin{equation}\label{eq:E} 
  E(q,\dot{q}) = \frac{1}{2}I\dot{q}^2 + U(q) \;,
\end{equation} 
where $U(q)=-\int_0^q \tau(s)\mathrm{d}s$ is the potential function. Clearly, $U(0)=0$ and $U(q)$ is smooth when seen as a function on $\R$. However, while $\tau(q)$ is a periodic function with the period of $2\pi$, $U(q)$ is not, as illustrated in Fig.~\ref{fig:Vq}. Thus, $U(q)$ undergoes a discontinuity at $q = 0$ when the function if defined on $\mathbb{S}$. In other words, 
\begin{displaymath}
 \lim_{q \to 0^+}U(q) \neq \lim_{q \to 2\pi^{-}}U(q) \;.
\end{displaymath}

We wish to derive conditions under which the system may exhibit a limit cycle. Since the state of the system is given by an angle and an angular velocity, the state space is a cylinder: $(q,\dot{q})\in \mathbb{S}\times \R$. The right-hand side of~\eqref{eq:model} is periodic w.r.t. $q$, which ensures the continuity of the vector field on the state space. We identify two types of closed orbits.
\begin{figure}
\centering
\includegraphics[width=0.3\columnwidth]{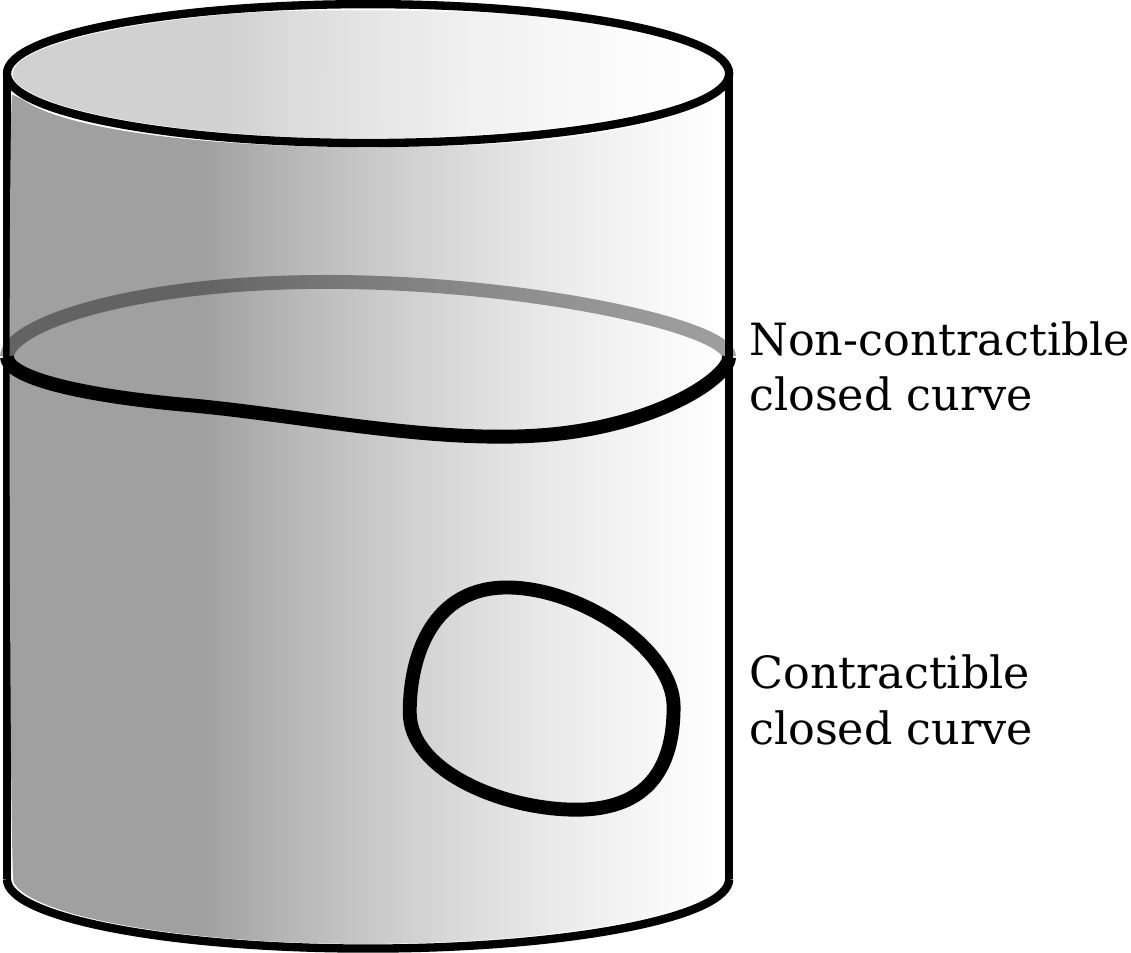}
\caption{Contractible and non-contractible curves on the cylinder.}
\label{fig:contract}
\end{figure}

\begin{defn} \label{def:contract}
  Let $\Gamma : [0,T] \to \mathbb{S}\times \R$ be a parametrized closed curve such that $\Gamma$ is injective 
  on $[0,T)$ and satisfies $\Gamma(0) = \Gamma(T)$. $T$ is called the \emph{period} of $\Gamma$. Suppose that 
  $\Gamma$ is a limit cycle. We say that $\Gamma$ is a \emph{contractible} limit cycle if it is homotopic to a 
  point $\gamma \in \mathbb{S}\times\R$. That is, if there exists a continuous function
  \begin{displaymath}
   h : [0,T]\times [0,1] \to \mathbb{S}\times \R 
  \end{displaymath}
  such that $h(t,0) = \Gamma(t)$ and $h(t,1) = \gamma$ for all $t \in [0,T]$. Otherwise, we say that 
 $\Gamma$ is a \emph{non-contractible} limit cycle. See Figure~\ref{fig:contract}.
\end{defn}

\begin{rem}
In certain applications, it is common to distinguish between {\em oscillatory} (i.e., contractible) and {\em rotational} (i.e., non-contractible) limit cycles, see, e.g.,~\cite[Def.\ 6]{Maggiore:18}. 
Alternatively, these are called the limit cycles of zeroth and first order (of homotopy) \cite{Barbashin:69}.
\end{rem}

The following theorem excludes the possibility of an oscillatory limit cycle in \eqref{eq:model}. The case of a rotational limit cycle is more involved and will be treated in detail below.
\begin{lem}\label{lem:l0}
 For $k_f>0$, system \eqref{eq:model} does not have an oscillatory limit cycle.
\end{lem}

\begin{proof}
 The total time-derivative of~\eqref{eq:E}, taken in virtue of \eqref{eq:model}, satisfies 
 \begin{equation}\label{eq:dE}
  \frac{\mathrm{d}}{\mathrm{d}t}E(q,\dot{q})=I\ddot{q}\dot{q}+\frac{\mathrm{d}}{\mathrm{d}q} U(q)\dot{q} =
   -k_f(\dot{q})^2+\tau(q)\dot{q}-\tau(q)\dot{q}=-k_f(\dot{q})^2 \le 0 \;.
 \end{equation}
 Let $\Gamma$ be a limit cycle. Compute the line integral of $\mathrm{d}E(q,\dot{q})$ along $\Gamma$:
 \begin{displaymath}
  -k_f \int_0^T(\dot{q}(s))^2 \mathrm{d}s = \oint_\Gamma \mathrm{d}E(q,\dot{q})=0 \;,
 \end{displaymath} 
where the last equality follows from continuity of the total energy $E(q,\dot{q})$ on $\mathbb{S}_0\times\R$ with $\mathbb{S}_0=\mathbb{S}\setminus \{0\}$. Thus we have that $\dot{q}(s) \equiv 0$, which contradicts the assumption.
\end{proof}

Before proceeding to the analysis of a rotational limit cycle, we write~\eqref{eq:model} as a system of two first-order 
ODEs,
\begin{equation}\label{eq:model2}
\begin{aligned}
 \dot{z}_1 &= z_2 \\
 \dot{z}_2 &= -\frac{k_f}{I} z_2+\frac{1}{I}\tau(z_1) \;,
\end{aligned}
\end{equation}
where $(z_1,z_2) = (q,\dot{q})$. In order to eliminate time, we divide the second differential equation by the first one and 
obtain
\begin{equation}\label{eq:model3}
 \frac{\mathrm{d}z_2}{\mathrm{d}z_1}=-\frac{k_f}{I}+\frac{1}{I}\frac{\tau(z_1)}{z_2} \;.
\end{equation}
System \eqref{eq:model2} has two null isoclines: $z_2=0$ (vertical inclination) and $z_2=\tau(z_1)/k_f$ (horizontal inclination). We have $\dot{z}_2>0$ below the second isocline and $\dot{z}_2<0$ above it. Also, we have that $\dot{z}_1>0$ when $z_2>0$ and $\dot{z}_1<0$ otherwise.

A non-contractible limit cycle can be represented as a graph of a periodic function, parametrized by $z_1$: $z^*_2(z_1):\mathbb{S}\mapsto \R$. We classify the limit cycles according to the sign of $z_2$.

\begin{defn}
A  non-contractible limit cycle $z_2^*(z_1)$ is said to be {\em sign semi-definite} if either $z^*_2(z_1)\ge 0$ or $z^*_2(z_1)\le 0$ for all $z_1\in[0,2\pi)$ and {\em sign definite} if the respective inequalities are strict. Otherwise, the limit cycle is said to be {\em sign-changing}.
\end{defn}

First we note that there cannot be any sign-changing  non-contractible limit cycle. If there was such a limit cycle, it would have to cross the axis $z_2=0$ at two regular (i.e., not equilibrium) points $z_1'\in[0,2\pi)$ and $z_1''\in[0,2\pi)$. Let $z_1''$ be the point where the trajectory changes from $z_2>0$ to $z_2<0$. We can immediately observe that the system trajectory cannot be extended beyond $z_1''$ as $z_2<0$ implies $\dot{z}_1<0$ and, hence, $z_1$ cannot increase any longer. This contradicts the definition of a non-contractible limit cycle. Note that the same argument can be used to show that the mapping $z_2^*(z_1)$ is one-to-one. On the other hand, there cannot exist a sign semi-definite limit cycle as the axis $z_2=0$ is a line of vertical inclination. Finally, we note that a non-contractible limit cycle cannot pass through an equilibrium point.

\begin{lem}\label{lem:l1}
Let $k_f\neq 0$ and $z_2^*(z_1)$ be a rotational limit cycle of \eqref{eq:model2}, 
i.e., $z_2^*(0)=z_2^*(2\pi)$. Then it holds that
\begin{displaymath}
 \int_0^{2\pi}z_2^*(z_1)\mathrm{d}z_1 = - \lim_{q \to 2\pi^{-}} \frac{ U(q)}{k_f} \;. 
\end{displaymath}
\end{lem}

\begin{proof}Substitute $z_2^*(z_1)$ into \eqref{eq:model3} and integrate
$$z_2^*(z_1)\frac{\mathrm{d}x^*_2(z_1)}{\mathrm{d}z_1}=-\frac{k_f}{I} z_2^*(z_1)+\frac{1}{I}\tau(z_1)$$
from $0$ to $2\pi$ to obtain
\begin{displaymath}
 0=-k_f\int_0^{2\pi}z_2^*(z_1)\mathrm{d}z_1- U(2\pi) \;,
\end{displaymath}
whence the result follows.
\end{proof}
\begin{lem}\label{lem:l2}
For $k_f\neq 0$, system \eqref{eq:model} cannot have more than one non-contractible limit cycle.
\end{lem}
\begin{proof}
Assume that $z^*_2(z_1)$ and $\bar{z}_2(z_1)$ are two non-contractible limit cycles. Then from Lemma~\ref{lem:l1} we have
\begin{displaymath}
 \int_0^{2\pi}\left[z_2^*(z_1)-\bar{z}_2(z_1)\right]\mathrm{d}z_1= \lim_{q \to 2\pi^{-}}\left[-\frac{ U(q)}{k_f}+\frac{ U(q)}{k_f} \right]= 0 \;.
\end{displaymath}
On the other hand, since the limit cycles cannot intersect, it should hold that either $z^*_2(z_1)>\bar{z}_2(z_1)$ or $z^*_2(z_1)<\bar{z}_2(z_1)$ for all $z_1\in[0,2\pi]$. This implies that $z^*_2(z_1)= \bar{z}_2(z_1)$.
\end{proof}

Finally, we can formulate the following important result:
\begin{thm}\label{thm:t2}For $k_f>0$, there exists at most one non-contractible limit cycle in \eqref{eq:model}. If such a limit cycle exists, it is sign definite with its sign opposite to the sign of $\lim_{q \to 2\pi^-}U(q)$. 
\end{thm}
\begin{proof}
This theorem follows from the previous analysis and Lemmata~\ref{lem:l1} and \ref{lem:l2}.
\end{proof}
These results are in accordance with the observed behavior of the system. For $\alpha\in(0,\pi)$, $\lim_{q\to 2\pi^{-}}U(q)> 0$ and hence the limit cycle, if it exists, lies below the horizontal axis. As $\alpha$ increases above $\pi$, the picture flips and the limit cycle appears in the upper half-plane.
\begin{figure}[tbh]
\centering
 \includegraphics[width=\textwidth]{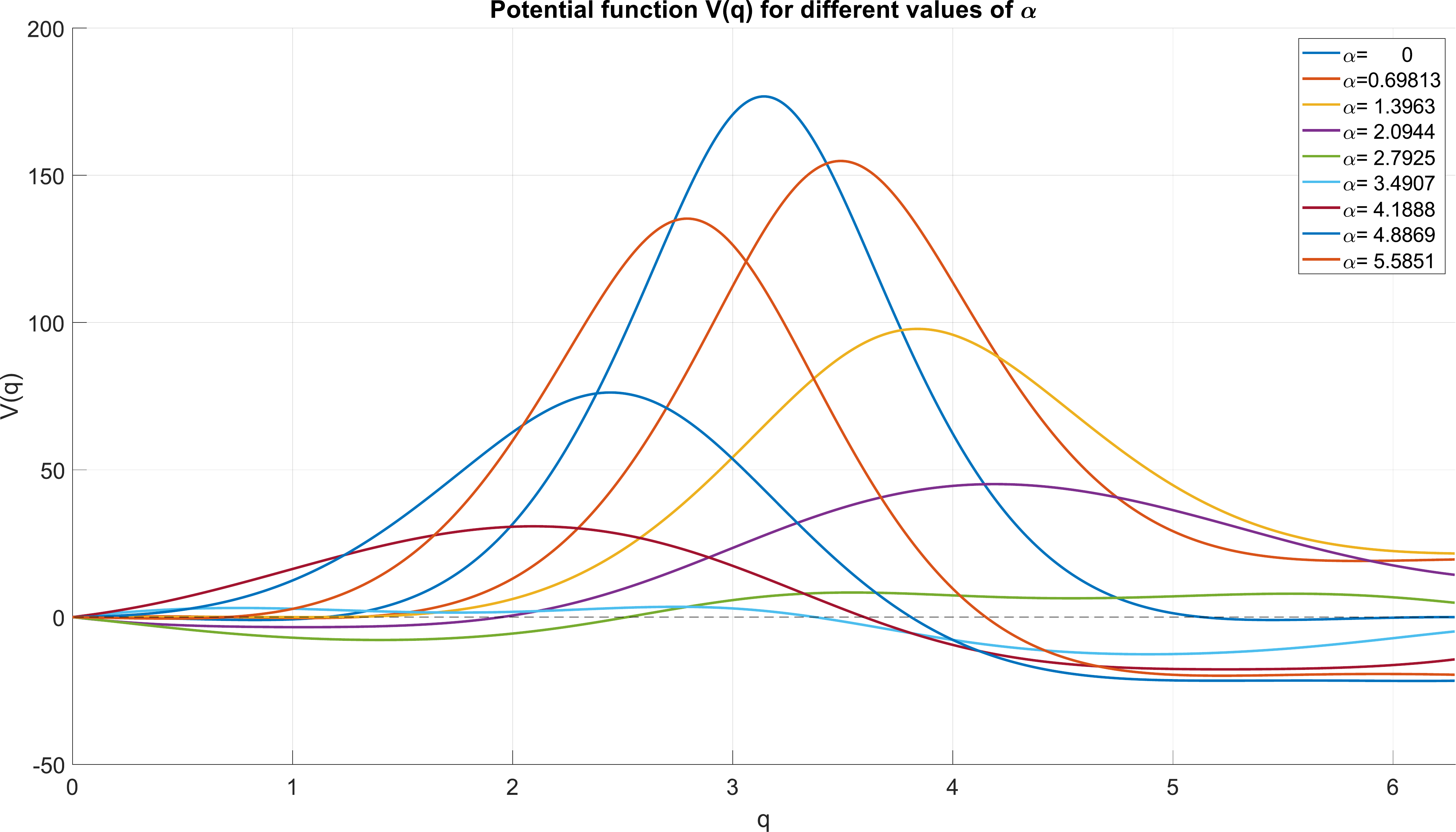}
 \caption{Potential functions $ U(q)$ for different values of $\alpha$ and $T_h=373.15$.}
\label{fig:Vq}
\end{figure}

\section{Numerical analysis} \label{sec:numerical}

In this section, we numerically determine the local and global bifurcations of~\eqref{eq:model}. There are basically two
definitions for the notion of \emph{bifurcation}. One emphasizes the loss of rank in the linearization of the vector field,
while the other is related to a qualitative change in its phase portrait. The first one is suitable 
for local analysis but fails to capture global features such as homoclinic bifurcations. We thus adhere 
to the second one~\cite{kuznetsov}. The idea of a qualitative change in the latter definition is formalized in
topological terms. More precisely, we say that two phase portraits are \emph{topologically equivalent} if there
exists a homeomorphism mapping orbits of one phase portrait to orbits of the other one. 
 
\begin{defn}
 The appearance of a topologically nonequivalent phase portrait under variation of parameters is called
 a \emph{bifurcation}.
\end{defn}

\begin{table}
\centering
\begin{tabular}{l|l|l}
 \textbf{Parameter}       & \textbf{Description}                              & \textbf{Value} \\
 \hline\hline
 $T_c$                    & Temperature of the cold bath                      & 298.15 K   \\           
 $T_h$                    & Temperature of the hot bath                       & 300 -- 500 K \\          
 $N$                      & Total amount of gas                               & 0.03 mol \\              
 $p_a$                    & Ambient pressure                                  & 100 kPa \\
 $R$                      & Universal gas constant                            & 8.314 J/$\mathrm{K}\cdot\mathrm{mol}$ \\
 \hline
 $V_{i,\max}$ & Maximum volume of the $i$th piston, $i=1,2$                   & 0.00046 $\mathrm{m}^3$ \\
 $A_i$             & Cross-sectional area of the $i$th piston, $i=1,2$        & 0.002 $\mathrm{m}^2$ \\
 $r$                      & Distance from shaft to the center of the flywheel & 0.1 m \\
 $l$                      & Length of the shaft                               & 0.3 m \\
 $\alpha$                 & Mechanical phase difference                       & 0 -- 2$\pi$ rad \\
 \hline
 $I$                      & Moment of inertia of the flywheel                 & 0.5 kg/$\mathrm{m}^2$ \\
 $k_f$                    & Friction coefficient                              & 0.1 $\mathrm{N}\cdot\mathrm{m}\cdot\mathrm{s}$/rad 
\end{tabular}
\caption{System parameters. The bifurcation parameters are $T_h$ and $\alpha$.}
\label{tab:param}
\end{table}

The system parameters are given in Table~\ref{tab:param}. We will carry out a two-parameter bifurcation
analysis with $T_h$ and $\alpha$ as the bifurcation parameters. In doing so, we will consider the complete
range for the phase angle, $\alpha \in [0,2\pi)$, while we will be only interested in a relatively low temperature
range, $T_h \in [300,500]$.

\subsection{Local bifurcations: Detection of equilibria and assessment of their stability}

The equilibria of~\eqref{eq:model} are necessarily of the form $(q^\star,0)$, where $q^\star$ is such that
$\tau(q^\star) = 0$. Finding $q^\star$ is a one-dimensional root-finding problem that can be easily solved numerically.
For each equilibrium point its stability is determined by linearization and computation of the eigenvalues. 
For an equilibrium $(q^\star,0)$, the Jacobian matrix of~\eqref{eq:model2} has the form 
\begin{displaymath}
  J=\begin{bmatrix} 0 & 1\\ -\tau'(q^\star)/I & -k_f/I \end{bmatrix} \;,
\end{displaymath}
where $\tau'(q^\star)$ is the derivative of $\tau(q)$ evaluated at $q^\star$. The respective characteristic function is $p(s)=Is^2+k_fs+\tau'(q^\star)$. Since $\tr (J) <0$, we conclude that there are no unstable foci. Thus, the system does not exhibit a bifurcation involving the crossing of the imaginary axis at complex values and the eventual appearance of the limit cycle cannot come from an Andronov-Hopf bifurcation. We also note that if such a limit cycle occurred, it would necessarily be of oscillatory, i.e., contractible type (since it surrounds the respective equilibrium point). However, this kind of limit cycle is ruled out by Lemma~\ref{lem:l0}.

Indeed, depending on the sign of $\tau'(q)$ at $q^\star$, a hyperbolic equilibrium point $(q^\star,0)$ is either a stable focus or a saddle. The non-hyperbolic equilibria correspond to triple zeros of the torque $\tau(q)$ and hence are either subcritical or supercritical pitchforks as illustrated in Fig.~\ref{fig:pitchforks}. 

\begin{figure}
\centering
 \includegraphics[width=0.5\columnwidth]{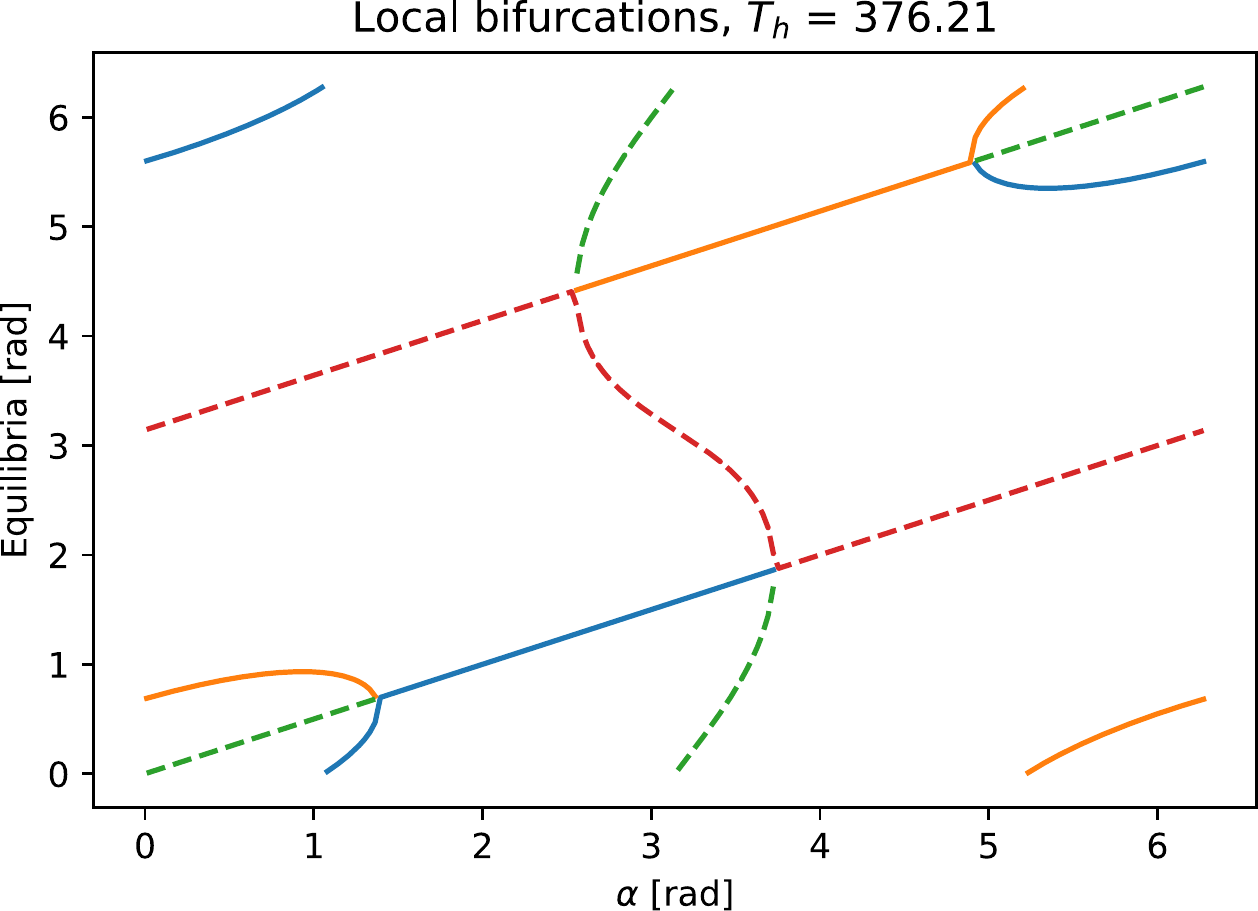}
 \caption{Local bifurcations of~\eqref{eq:model}. Solid lines correspond to stable foci and dashed lines correspond
  to saddles. The system exhibits pitchfork bifurcations of codimension one. We identify the points $\alpha = 0$ and 
  $\alpha=2\pi$, as well as the points $q=0$ and $q = 2\pi$, so the diagram should be visualized on the torus $\mathbb{S}^1\times \mathbb{S}^1$.}
\label{fig:pitchforks}
\end{figure}

An evenly spaced subset of the parameter space $\DD = [0,2\pi)\times[300,500]$ was chosen and the equilibria were computed
for each pair $(\alpha,T_h)$. The type of equilibrium (stable focus or saddle) was determined according to the sign of $\tau'(q^\star)$. Figure~\ref{fig:pitchforks}
shows the one-parameter bifurcation diagram for a fixed temperature $T_h = 376.2$ as $\alpha$ varies. The figure shows that 
the set of equilibria transitions from a set having two stable foci and two saddles to a set having only one saddle and 
one stable focus. The system transitions back and forth at four pitchforks, two subcritical and two supercritical. 
Note that the diagram is symmetrical with respect to the transformation $(q,\alpha) \mapsto (2\pi-q,2\pi-\alpha)$.

\begin{figure}[tbh]
\centering
 \includegraphics[width=0.5\columnwidth]{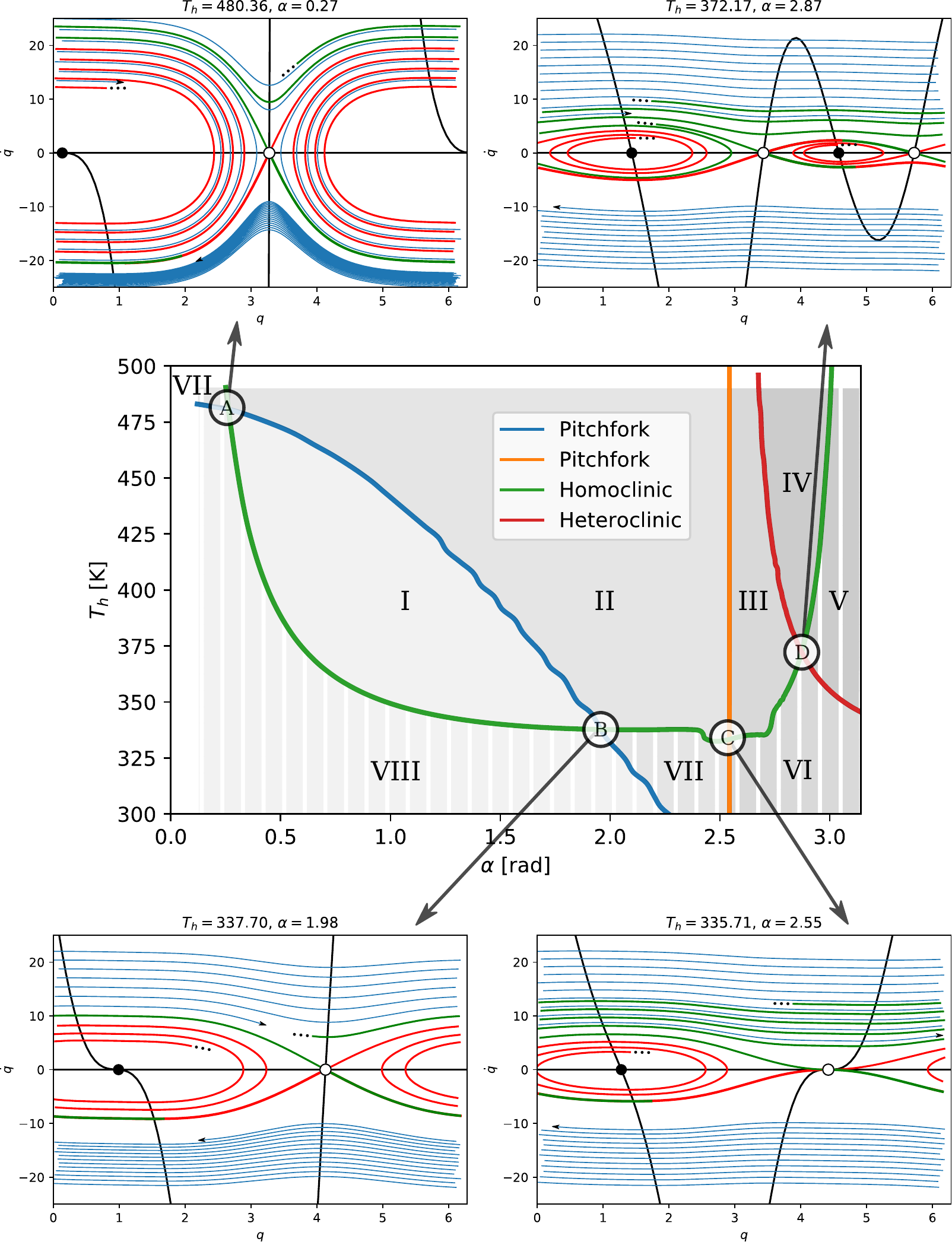}
 \caption{Bifurcation curves. The homoclinic and heteroclinic bifurcation curves are indicated by green and red lines. The occurrence 
   pitchfork bifurcations are indicated blue and orange lines. The bifurcation curves separate the parameter space into 
   eight regions, denoted by roman numerals. The intersection of bifurcation lines correspond to bifurcations of codimension
   two, denoted by letters A, B and C.}
\label{fig:complete}
\end{figure}

Before continuing, allow us to recall the following~\cite{kuznetsov}.
\begin{defn}
  The \emph{codimension} of a bifurcation in a system is the difference between the dimension of the parameter
  space and the dimension of the corresponding bifurcation boundary.
\end{defn}

\begin{figure}[tbh]
\centering
 \includegraphics[width=0.5\columnwidth]{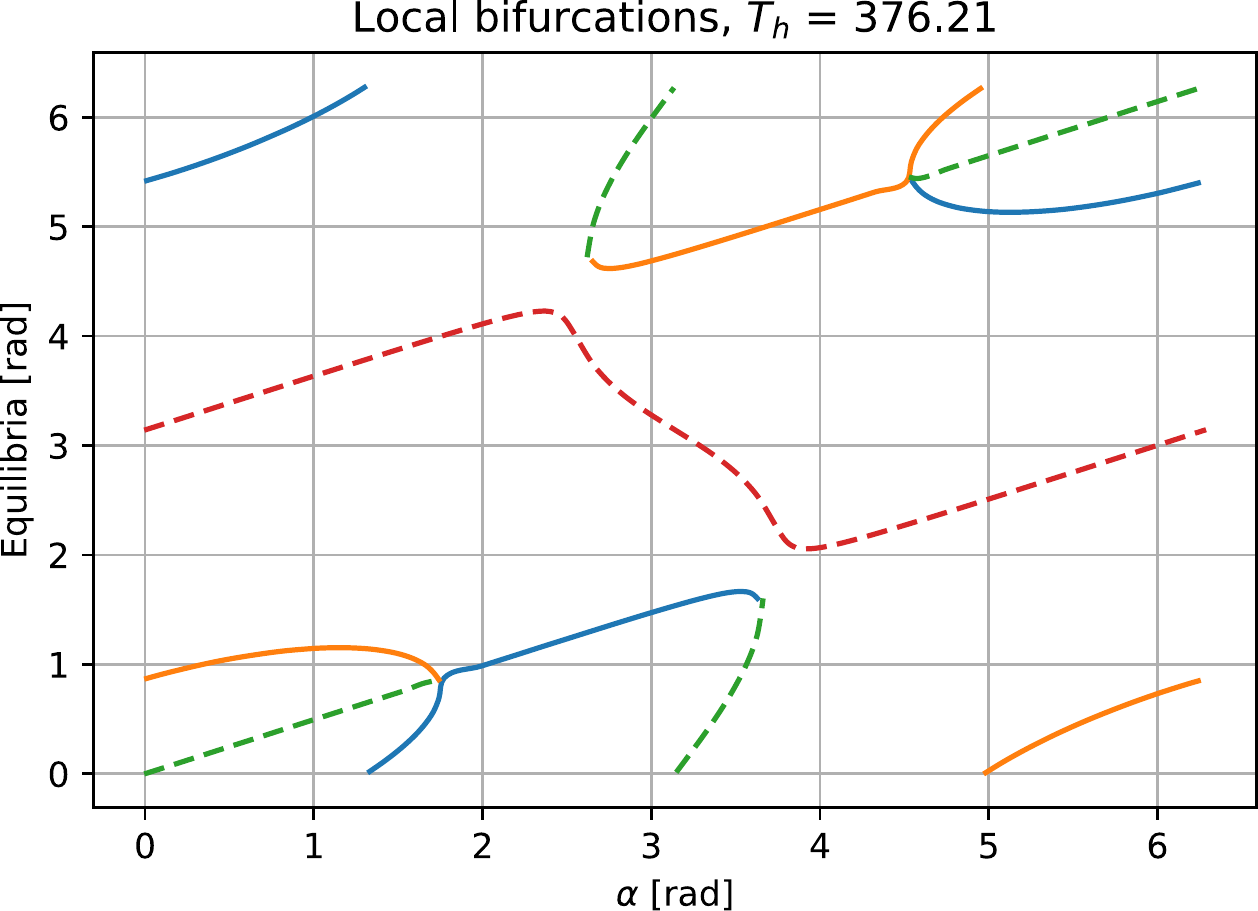}
 \caption{Local bifurcations of~\eqref{eq:model} when $A_2$ is perturbed from $2e{-}3$ to $2.05e{-}3$ $\mathrm{m}^2$.
  Two of the pitchforks shown in Fig.~\ref{fig:pitchforks} breakdown into folds.}
\label{fig:folds}
\end{figure}

Figure~\ref{fig:complete} shows the occurrence of the pitchfork bifurcations in the two-dimensional parameter space
(blue and orange lines). 
The locus of the pitchforks is one-dimensional, so the codimension of these bifurcations is, according to this 
definition, equal to one. We note, however, that the two subcritical pitchforks in the center are not structurally stable and break into folds when either of the symmetry conditions $A_1 = A_2$ or $V_1 = V_2$ is infringed (cf. Fig.~\ref{fig:folds}). 

\subsection{Global bifurcations: Continuation of homoclinic and heteroclinic bifurcations.}

\begin{figure}[tbh]
\centering
 \includegraphics[width=0.5\columnwidth]{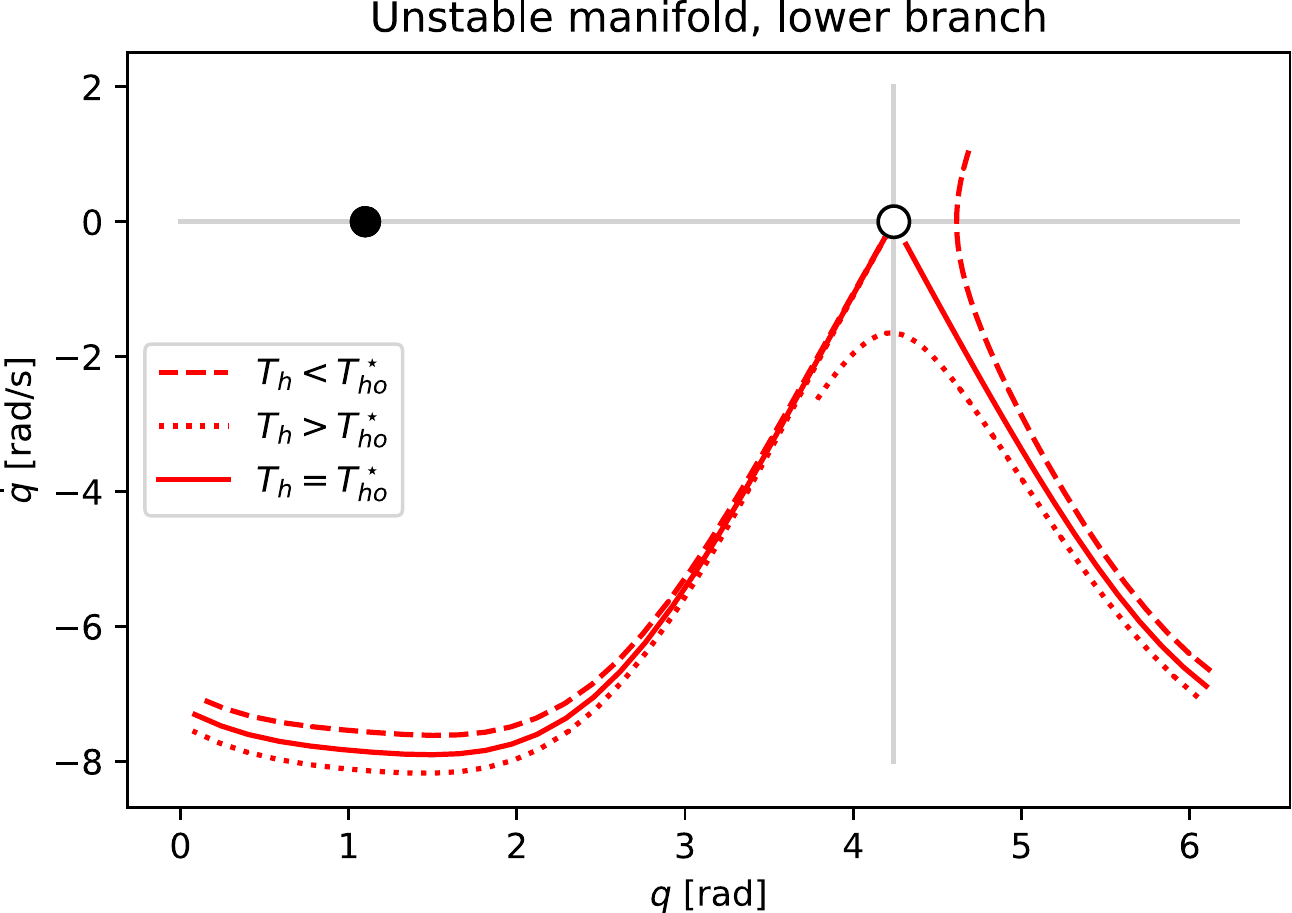}
 \caption{For a fixed $\alpha = 2.2$, a homoclinic orbit appears at the bifurcation value $T_{ho}^\star = 337.6$. 
  That is, the closure of one of the branches of the saddle's unstable manifold is a closed orbit. For lower temperatures,
  the focus is included in the closure of the unstable manifold. For higher temperatures, the closure contains 
  the orbit of a limit cycle (not shown).}
\label{fig:bisectionHomo}
\end{figure}

For practical purposes, the homoclinic bifurcation is the most important one, as it establishes
the minimal operating conditions for the engine. Regarding $\alpha$, we will only consider the
interval $[0,\pi)$, as the system behavior in the complement $[\pi, 2\pi)$ can be easily
determined from symmetry. We fix $\alpha$ and then find the temperature $T_{ho}^\star$ at which the
system exhibits a homoclinic orbit. To see how $T_{ho}^\star$ can be found, consider
Fig.~\ref{fig:bisectionHomo} and note that, for $T_{h} < T_{ho}^\star$, a trajectory going along 
the unstable manifold crosses the horizontal axis first and then converges to the focus. For $T_{h} > T_{ho}^\star$,
the trajectory crosses first a vertical axis centered on the saddle and then converges to a limit cycle. This
suggests the `test function'
\begin{displaymath}
  \psi_1(T_{h}) =
  \begin{cases}
    1 & \text{if } t_f < t_c \\
   -1 & \text{if } t_f > t_c
  \end{cases} \;.
\end{displaymath}
Here, $t_f = \min \left\{ t \mid \dot{q}(t;q_0,\dot{q}_0) = 0 \right\}$ and
$t_c = \min \left\{ t \mid q(t;q_0,\dot{q}_0) = q^\star \right\}$, and $q(t;q_0,\dot{q}_0)$
denotes the trajectory initiating at
\begin{displaymath}
  \begin{pmatrix}
   q_0 \\ \dot{q}_0
  \end{pmatrix}
  = 
  \begin{pmatrix}
   q^\star \\ 0
  \end{pmatrix} + \varepsilon \begin{cases}-v& \mbox{for}\; \alpha\in[0,\pi)\\v& \mbox{for}\; \alpha\in[\pi,2\pi)\end{cases}
\end{displaymath}
with $q^\star$ the angular component of the saddle located closer to $q=\pi$ and $v$ the unit eigenvector associated to the unstable 
eigenvalue of the system linearized at $q^\star$ and chosen in the way that its $z_2$-component is positive. 
By convention, $\min \{\emptyset\} = \infty$. 
Simply put, $\psi_1(T_{h}) = 1$ if the trajectory crosses the horizontal axis first and $\psi(T_{h}) = -1$
if the trajectory crosses the vertical axis first. 

Given two values of $T_{h}$ whose images under $\psi_1$ have different signs, it is easy to determine $T_{ho}^\star$ 
using a bisection or secant algorithm with a fixed number of iterations. The same process is repeated for different values
of $\alpha$ to construct the continuation curve shown in Fig.~\ref{fig:complete} (green line). For every
pair of parameters below the curve, that is, lying inside the hatched area, the engine does not turn. More formally: Except for a 
zero-measure set of initial conditions, every trajectory converges to a focus. If, on the other hand, 
a pair of parameters lies above the curve, then there exists a set of initial conditions having positive measure 
and such that all trajectories starting there converge to the cycle. The bifurcation is similar to
the Andronov-Leontovich bifurcation that occurs on Euclidean spaces~\cite[Ch. 6]{kuznetsov}, except that 
the bifurcations that occur in~\eqref{eq:model} lead to a contractible instead of a non-contractible limit cycle.

\begin{figure}
\centering
 \includegraphics[width=0.5\columnwidth]{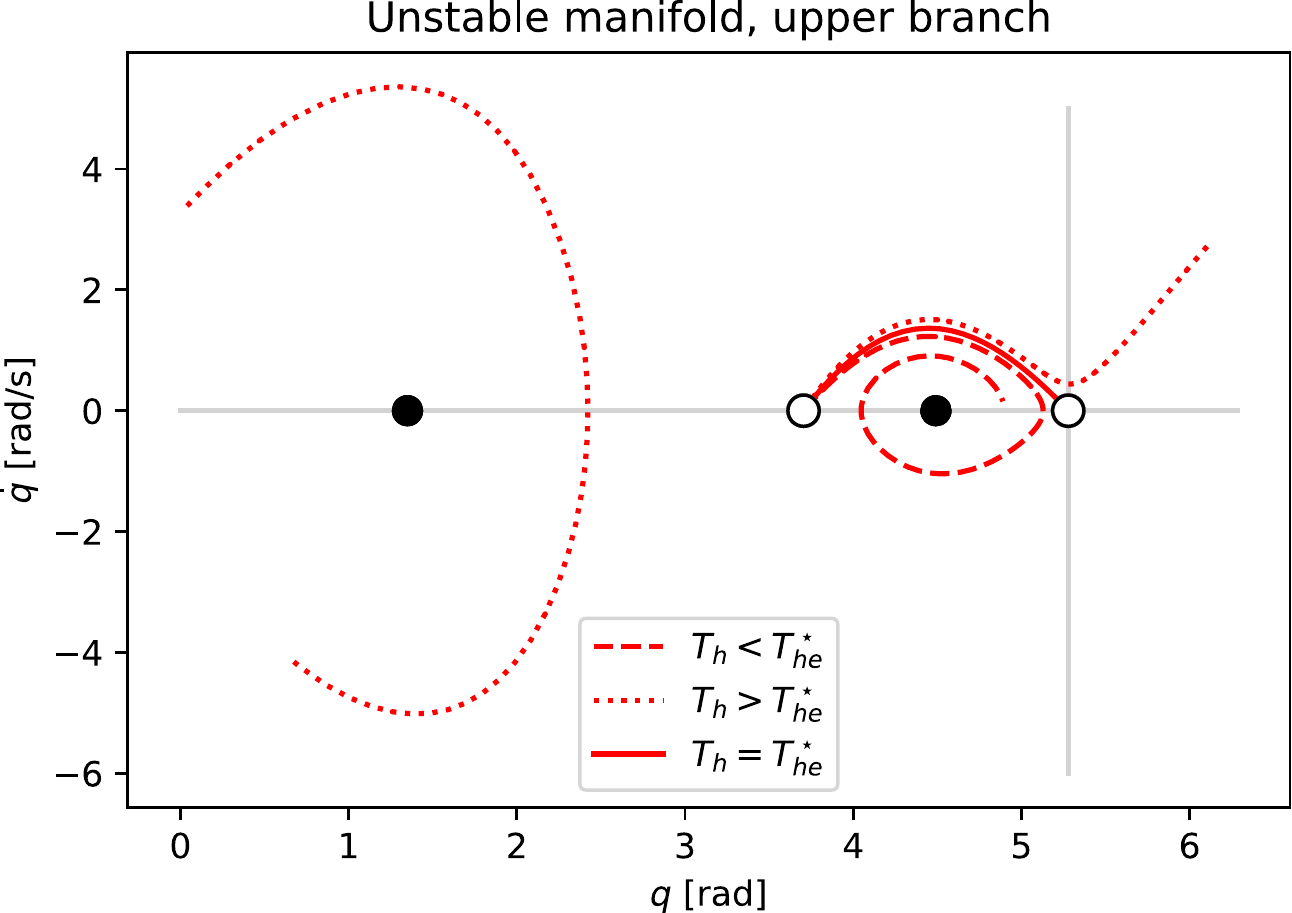}
 \caption{For a fixed $\alpha = 2.6$, a heteroclinic orbit connecting two saddles appears at the bifurcation value
  $T_{he}^\star = 451.8$. For lower temperatures, the focus sitting on the right is included in the closure of the
  unstable manifold. For higher temperatures, the closure contains the focus on the left.}
\label{fig:bisectionHetero}
\end{figure}

We will shortly see that there are parameters for which the corresponding phase plane has a heteroclinic orbit
connecting two saddles. Similar to the detection of homoclinic orbits, we fix $\alpha$ and then find the
temperature $T_{he}^\star$ at which the system exhibits the heteroclinic orbit. To see how $T_{he}^\star$ 
can be found, consider now Fig.~\ref{fig:bisectionHetero} and note that, for $T_{h} < T_{he}^\star$, a trajectory 
going along the unstable manifold of the left saddle, $q^\star$, crosses first the horizontal axis at 
the right of the saddle, while for $T_{h} > T_{he}^\star$, the trajectory crosses the vertical line at $q^{\star\star}$
with $q^{\star\star}$ the angular coordinate of the second saddle. 
This suggests the test function
\begin{displaymath}
  \psi_2(T_{h}) =
  \begin{cases}
    1 & \text{if } t_f < t_u \\
   -1 & \text{if } t_f > t_u
  \end{cases} \;.
\end{displaymath}
where $t_f$ is defined as above, $t_u = \min \left\{ t \mid q(t;q_0,\dot{q}_0) = q^{\star\star} \right\}$, and $q(t;q_0,\dot{q}_0)$
denotes the trajectory initiating at
\begin{displaymath}
  \begin{pmatrix}
   q_0 \\ \dot{q}_0
  \end{pmatrix}
  = 
  \begin{pmatrix}
   q^\star \\ 0
  \end{pmatrix} + \varepsilon \begin{cases}v& \mbox{for}\; \alpha\in[0,\pi)\\-v& \mbox{for}\; \alpha\in[\pi,2\pi)\end{cases}
\end{displaymath}

As before, given two values of $T_{h}$ whose images under $\psi_2$ have different signs, it is easy to determine $T_{he}^\star$ 
using a bisection or secant algorithm with a fixed number of iterations. The resulting continuation curve is shown in 
Fig.~\ref{fig:complete} (red line). It can be seen that the bifurcation is also of codimension one.

\subsection{Higher codimension bifurcations}

Fig.~\ref{fig:complete} shows that the bifurcation curves divide the parameter space into eight regions.
Within each region, every pair of parameters correspond to the same topologically equivalent class of phase planes. 
The bifurcation curves intersect transversely and, at the intersections, we can find bifurcations of codimension two.
These are highly degenerate conditions, unlikely to manifest physically, but which are useful from an analytic 
point of view, inasmuch as several nonequivalent phase planes can be obtained by arbitrarily small perturbations of
the parameters. 

It is interesting to note that, when both the homoclinic and the heteroclinic orbits
coexist (point D in Fig.~\ref{fig:complete}), they both originate from the same saddle, but each belongs to a different branch of the unstable manifold.

\begin{figure}
\centering
 \includegraphics[width=0.95\columnwidth]{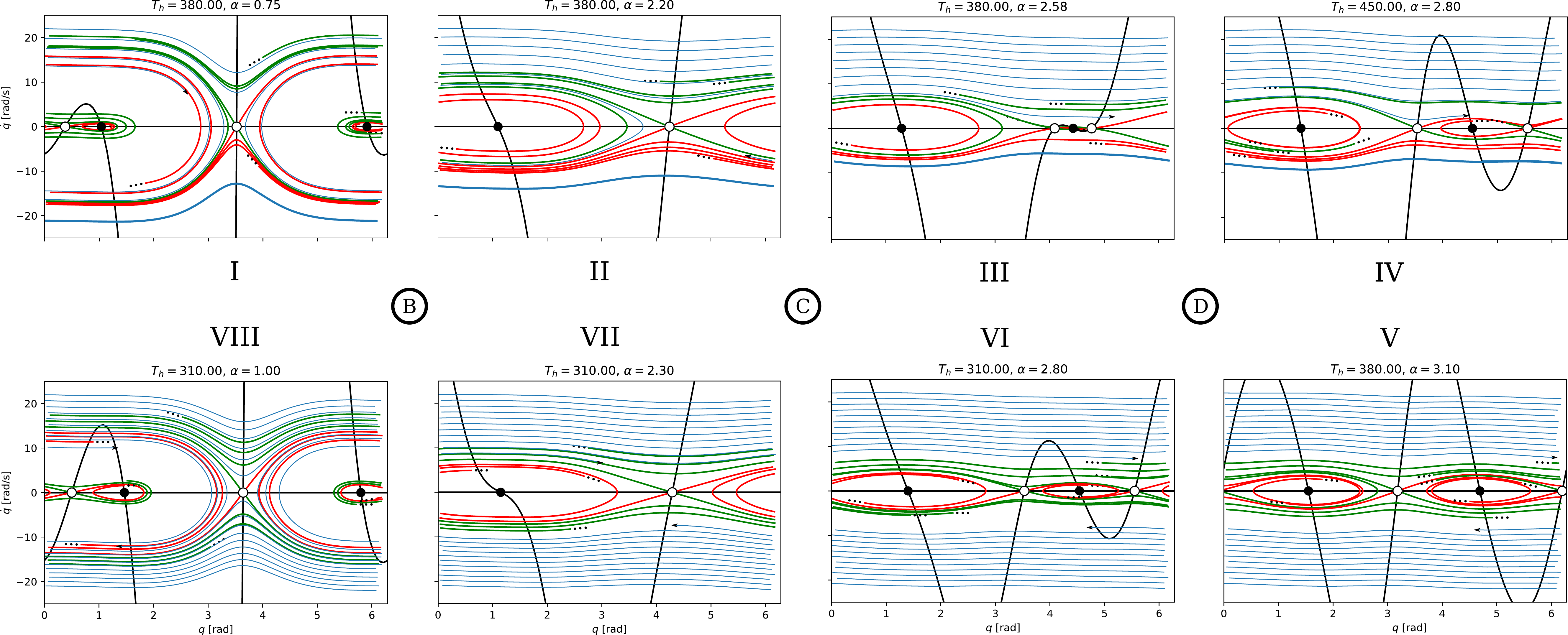}
 \caption{Possible generic behaviors of~\eqref{eq:model}. Regions I, II, VII and VIII in Fig.~\ref{fig:complete} 
  can be obtained as perturbations of the parameter pair denoted by B. Regions III, IV, V and VI can be obtained 
  as perturbations of D. 
  }
\label{fig:regionsBCD}
\end{figure}

Figure~\ref{fig:regionsBCD} shows four nonequivalent phase planes that result from perturbations of the parameter pairs 
denoted by B, C and D in Fig.~\ref{fig:complete}. Except for the bifurcation C, \emph{all} qualitatively different
behaviors can be obtained as perturbations of the bifurcations B and D. It can be observed that the bifurcations A
and B are actually the same (they are topologically equivalent).

\subsection{Output power}

Figure~\ref{fig:work} shows the pressure and total volume of the cylinders as $q$ varies from 0 to $2\pi$.
The area inside the curve corresponds to the work delivered on one cycle $\Gamma$,
\begin{displaymath}
 W = \oint_\Gamma p\mathrm{d} V = \int_0^{2\pi}p(q)V'(q)\mathrm{d}q \;.
\end{displaymath}
For a given pair of parameters $(\alpha,T_h)$ we compute $W$ by differentiating $V$ symbolically and 
integrating the integrand numerically. To compute the average power we require the period, $T$, of the limit cycle, if it
exists. The problem of finding $T$ is formulated as a boundary value problem. First, time is rescaled as
$\tau = t/T$ so that, in the new time scale, the period is fixed to 1. The differential
equation~\eqref{eq:model2} takes the form
\begin{equation} \label{eq:bvp}
\begin{split}
 \frac{\mathrm{d} z_1}{\mathrm{d} \tau} &= T z_2 \\
 \frac{\mathrm{d} z_2}{\mathrm{d} \tau} &= \frac{T}{I}\left( -k_f z_2 + \tau(z_1) \right)
\end{split} \;.
\end{equation}
The boundary value problem is that of finding $T$ and a solution to~\eqref{eq:bvp}, subject to the
boundary conditions
\begin{displaymath}
 z_1(0) = 0 \;, \quad z_1(1) = 2\pi \quad \text{and} \quad z_2(0) = z_2(1) \;.
\end{displaymath}
The problem can be solved using standard software.

\begin{figure}[tbh]
\centering
 \includegraphics[width=0.4\columnwidth]{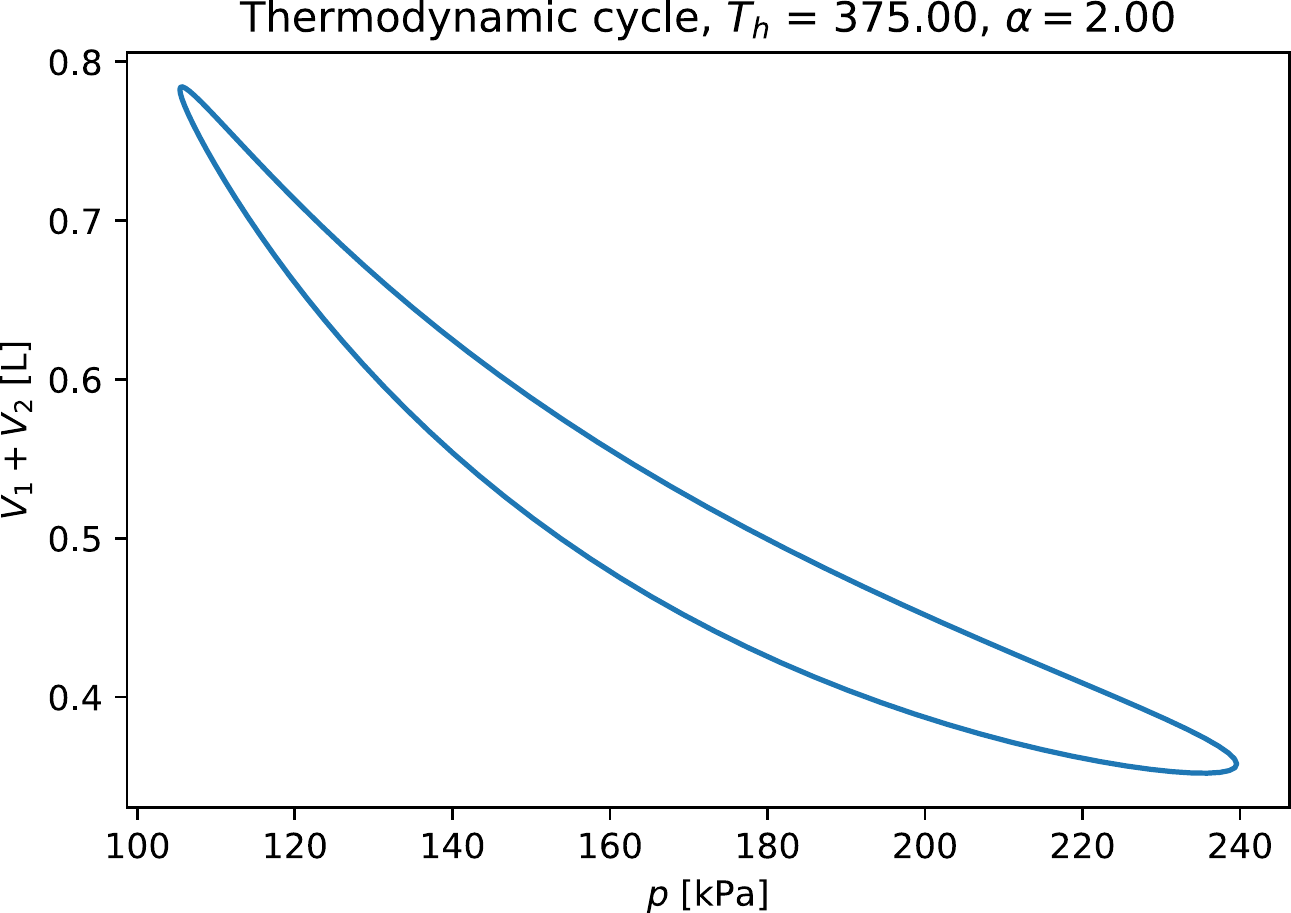}
 \caption{Pressure vs. volume as $q$ varies from 0 to $2\pi$. The area inside the curve corresponds to
  the work delivered by an engine during one cycle}
\label{fig:work}
\end{figure}

\begin{figure}[tbh]
\centering
 \includegraphics[width=0.6\columnwidth]{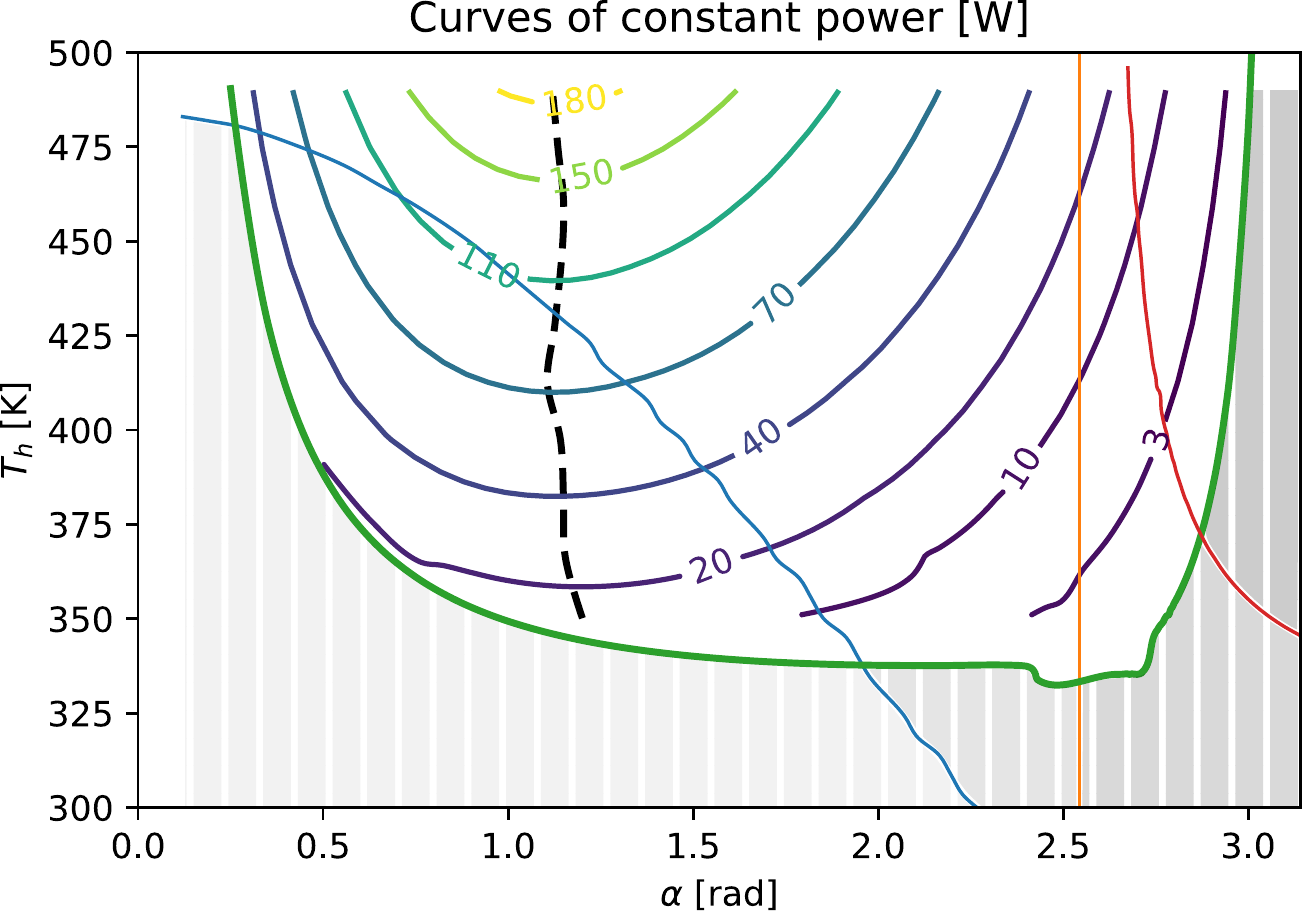}
 \caption{Contour plot of the average output power 
  $\hat{P}(\alpha,T_h) = \frac{1}{T} \int_0^{2\pi}p(q)V'(q)\mathrm{d}q$. Bifurcation curves are indicated
  as in Fig.~\ref{fig:complete}. The locus indicated by the dashed line consists of points of the form
  $(\alpha^\star, T_h)$ with $\alpha^\star = \argmax_{\alpha} \hat{P}(\alpha,T_h)$.}
\label{fig:power}
\end{figure}

The average power $\hat{P}(\alpha,T_h) = W/T$ was computed for different pairs $(\alpha,T_h)$. A contour plot is shown in Fig.~\ref{fig:power}.
As expected, the output power approaches zero as the parameters approach the homoclinic bifurcation curve. Based on the performed analysis we can make a practically important observation: for considered values of parameters, the value of $\alpha$ corresponding to the maximal average power is nearly constant and is approximately equal to $1.2$ (see the dashed line in Fig.\ \ref{fig:power}). 
%




\section{Conclusions}

In this paper, we considered an interconnected system consisting of a Stirling engine and a mechanical load (a flywheel). The mechanism of the appearance of a limit cycle is studied in detail and a thorough bifurcation analysis of the considered thermo-mechanic system is carried out, both analytically and numerically. It has been shown that for each particular value of the phase shift between the pistons positions, $\alpha$, there is a critical hot source temperature that corresponds to the occurrence of a limit cycle. Furthermore, the study reveals that $\alpha$ -- which is typically chosen according to empirical considerations -- can be used to optimize the system output power. To the best of our knowledge, this is the first result aimed at determining the optimal value of the phase shift.

The present research will be extended along the following lines. Firstly, a more elaborate model of the Stirling engine will be employed to verify the results obtained for the Schmidt isothermal model. Secondly, a detailed sensitivity analysis will be carried out in order to estimate the range of the parameters values for which the described phenomena take place.
Given that there are eight different generic phase planes, it is not unreasonable to conjecture the existence of a bifurcation
of codimension three such that all the possible behaviors of the engine are obtained as perturbations of this
single bifurcation. A third line of research is thus to enlarge the dimension of the bifurcation parameter space and investigate if
such a degenerate bifurcation exists. If this is indeed the case, the model can be considerably simplified without loosing the
qualitative properties of the original one.

\bibliography{Stirling}          

\end{document}